\documentclass[12pt]{amsart}

\usepackage{ucs}

\usepackage{amssymb}
\usepackage{amsthm}
\usepackage{amsmath}
\usepackage{latexsym}
\usepackage[cp1251]{inputenc}
\usepackage[mathcal]{eucal}
\usepackage{graphicx}
\usepackage{wrapfig}
\usepackage{caption}
\usepackage{subcaption}
\usepackage{indentfirst}
\usepackage[left=2.6cm,right=2.6cm,top=3cm,bottom=3cm,bindingoffset=0cm]{geometry}
\usepackage{enumerate}

\DeclareMathOperator{\aut}{Aut}

\DeclareMathOperator{\cay}{Cay}
\DeclareMathOperator{\cyc}{Cyc}

\DeclareMathOperator{\iso}{Iso}

\DeclareMathOperator{\orb}{Orb}

\DeclareMathOperator{\rk}{rk}

\DeclareMathOperator{\Span}{Span}

\DeclareMathOperator{\sym}{Sym}
\DeclareMathOperator{\rad}{rad}

\DeclareMathOperator{\alg}{Alg}

\makeatletter 
\def\@seccntformat#1{\csname the#1\endcsname. } 
\def\@biblabel#1{#1.}

\makeatother

\title{Separability of Schur rings over an abelian group of order~$4p$}

\author{Grigory Ryabov}
\address{Sobolev Institute of Mathematics, Novosibirsk, Russia}
\address{Novosibirsk State University, Novosibirsk, Russia}
\email{gric2ryabov@gmail.com}
\thanks{The work is supported by the Russian Foundation for Basic Research (project 18-31-00051)}

\date{}

\newtheorem{prop}{Proposition}[section]

\newtheorem*{theo1}{Theorem 1}

\newtheorem{lemm}[prop]{Lemma}

\newtheorem*{corl1}{Corollary}
\theoremstyle{definition}

\begin{document}

\vspace{\baselineskip}
\vspace{\baselineskip}

\vspace{\baselineskip}

\vspace{\baselineskip}

\begin{abstract}
An $S$-ring (a Schur ring) is said to be \emph{separable} with respect to a class of groups $\mathcal{K}$ if every  its algebraic isomorphism to an $S$-ring over a group from $\mathcal{K}$ is induced by a combinatorial isomorphism. We prove that every Schur ring over an abelian group $G$ of order $4p$, where $p$ is a prime, is separable with respect to the class of abelian groups. This implies that the Weisfeiler-Leman dimension of the class of Cayley graphs over~$G$ is at most~2.
\\
\\
\textbf{Keywords}: Schur rings, Cayley graphs, Cayley graph isomorphism problem.
\\
\textbf{MSC}:05E30, 05C60, 20B35.
\end{abstract}

\maketitle

\section{Introduction}

An \emph{$S$-ring} (a \emph{Schur ring}) over a finite group $G$ is defined to be a subring of the integer group ring $\mathbb{Z}G$ that is a free $\mathbb{Z}$-module spanned by a partition of $G$ closed under taking inverse and containing the identity element  of $G$ as a class (for exact definitions see Section~$2$). The elements of this partition  are called the \emph{basic sets} of the $S$-ring. The theory of $S$-rings was initiated by Schur \cite{Schur} and later developed by Wielandt \cite{Wi}. For more details on $S$-rings see~\cite{MP}.

Let $\mathcal{A}$ and $\mathcal{A}^{'}$ be $S$-rings over groups $G$ and $G^{'}$ respectively. \emph{A (combinatorial) isomorphism } from $\mathcal{A}$ to $\mathcal{A}^{'}$  is defined to be a bijection  $f:G\rightarrow G^{'}$ such that for every basic set $X$ of $\mathcal{A}$ the set $X^{'}=X^f$ is a basic set of $\mathcal{A}^{'}$ and $f$ is an isomorphism of the Cayley graphs $\cay(G,X)$ and 
$\cay(G^{'},X^{'})$. \emph{An algebraic isomorphism} from  $\mathcal{A}$ to $\mathcal{A}^{'}$ is defined to be a ring isomorphism of them  inducing the bijection between the basic sets of $\mathcal{A}$ and  the basic sets of $\mathcal{A}^{'}$. One can check that every combinatorial isomorphism induces an algebraic isomorphism. However, not every algebraic isomorphism is induced by a combinatorial one (see~\cite{EP1}).

Let $\mathcal{K}$ ba a class of groups. An $S$-ring is said to be \emph{separable} with respect to $\mathcal{K}$ if every algebraic isomorphism from it to an $S$-ring over a group from $\mathcal{K}$ is induced by a combinatorial isomorphism. A separable $S$-ring is determined up to isomorphism only by the tensor of its structure constants. A finite group  is said to be \emph{separable} with respect to $\mathcal{K}$ if every $S$-ring over this group is separable with respect to $\mathcal{K}$. Denote the classes of cyclic and abelian groups by $\mathcal{K}_C$ and $\mathcal{K}_A$ respectively. In~\cite{EP5} it was proved that cyclic $p$-groups are separable with respect to $\mathcal{K}_C$. On the other hand, there are examples of cyclic groups which are not separable with respect to $\mathcal{K}_C$~\cite{EP1}. Denote the cyclic group of order~$n$ by $C_n$. The results obtained in~\cite{Ry} imply that the groups $C_{p^k}$ and $C_p\times C_{p^k}$, where $p\in\{2,3\}$ and $k\geq 1$, are separable with respect to $\mathcal{K}_A$. However, the classification of all  separable groups is far from complete. In fact, only the above  families of groups are known infinite families of separable groups.

In this paper we study $S$-rings and abelian groups which are separable with respect to $\mathcal{K}_A$. Throughout the paper we write for short ``separable''   instead ``separable with respect to $\mathcal{K}_A$''. The main result of the paper is given in the theorem below.

\begin{theo1}\label{main}
An abelian group of order $4p$ is separable for every prime $p$.
\end{theo1}

Let $G$ be an abelian group of order $4p$, where $p$ is a prime. Then $p=2$ and $G\cong C_4 \times C_2$, or $G\cong C_{4p}$, or $G\cong C_2\times C_2\times C_p$. If $G\cong C_4 \times C_2$ then Theorem~\ref{main} follows from~\cite[Theorem 1]{Ry}. The proof of Theorem~\ref{main} for other groups is based on the description of $S$-rings over $G$ that was obtained in \cite{EP3} for $G\cong C_{4p}$ and in \cite{EKP} for $G\cong C_2\times C_2\times C_p$. We give this description in a form convenient for us in Section~$3$.

A motivation for being interested in  separable  groups comes from the problem of testing isomorphism of Cayley graphs. If a group~$G$ is separable then  the isomorphism problem for Cayley graphs over~$G$ can be solved efficiently by using the Weisfeiler-Leman algorithm  \cite{WeisL}. In the sense of~\cite{KPS} this means that the Weisfeiler-Leman dimension of the class of Cayley graphs over~$G$ is at most~2. For more details see also \cite[Section~6.2]{EP2} and \cite[Section~8]{Ry}.

\begin{corl1}\label{corollary}
Let $p$ be a prime, $G$ an abelian group of order $4p$, and $\mathcal{G}$ the class of Cayley graphs over~$G$. Then the Weisfeiler-Leman dimension of $\mathcal{G}$ is at most~2. 
\end{corl1}

\begin{proof}
Follows from \cite[Proposition~8.1]{Ry} and Theorem~\ref{main}.
\end{proof}

It should be mentioned that the recognition and the isomorphism problems for Cayley graphs over an abelian group of order $4p$, where $p$ is a prime, were solved in~\cite{NP}.\medskip

{\bf Notation.}

As usual by $\mathbb{Z}$ we denote the ring of rational integers.

The projections of $X\subseteq A\times B$ to $A$ and $B$ are denoted by $X_A$ and $X_B$, respectively.

The set of non-identity elements of a group $G$ is denoted by  $G^\#$.

Let $X\subseteq G$. The element $\sum_{x\in X} {x}$ of the group ring $\mathbb{Z}G$ is denoted by $\underline{X}$.

The order of $g\in G$ is denoted by $|g|$.

The set $\{x^{-1}:x\in X\}$ is denoted by $X^{-1}$.

The subgroup of $G$ generated by $X$ is denoted by $\langle X\rangle$; we also set $\rad(X)=\{g\in G:\ gX=Xg=X\}$.

If  $m\in \mathbb{Z}$ then the set $\{x^m: x \in X\}$ is denoted by $X^{(m)}$.

The set of edges of the Cayley graph $\cay(G,X)$ is denoted by $R(X)$.

The group of all permutations of $G$ is denoted by $\sym(G)$.

The subgroup of $\sym(G)$ induced by right multiplications of $G$ is denoted by $G_{right}$.

For a set $\Delta\subseteq \sym(G)$ and a section $S=U/L$ of $G$ we set 
$$\Delta^S=\{f^S:~f\in \Delta,~S^f=S\},$$
where $S^f=S$ means that $f$ permutes the $L$-cosets in $U$ and $f^S$ denotes the bijection of $S$ induced by $f$.

If a group $K$ acts on a set $X$ then the set of all orbtis of $K$ on $X$ is denoted by $\orb(K,X)$.

The cyclic group of order $n$ is denoted by  $C_n$.

\section{Preliminaries}

In this section we use the notation and terminology from paper~\cite{Ry}, where the most part
of the material is contained.

\subsection{Definitions}
Let $G$ be a finite group and $\mathbb{Z}G$  the integer group ring. Denote the identity element of $G$ by $e$.  A subring  $\mathcal{A}\subseteq \mathbb{Z} G$ is called an \emph{$S$-ring} over $G$ if there exists a partition $\mathcal{S}=\mathcal{S}(\mathcal{A})$ of~$G$ such that:

$(1)$ $\{e\}\in\mathcal{S}$,

$(2)$  if $X\in\mathcal{S}$ then $X^{-1}\in\mathcal{S}$,

$(3)$ $\mathcal{A}=\Span_{\mathbb{Z}}\{\underline{X}:\ X\in\mathcal{S}\}$.

The elements of $\mathcal{S}$ are called the \emph{basic sets} of  $\mathcal{A}$ and the number $|\mathcal{S}|$ is called the \emph{rank} of~$\mathcal{A}$. If $X,Y,Z\in\mathcal{S}$ then   the number of distinct representations of $z\in Z$ in the form $z=xy$ with $x\in X$ and $y\in Y$ is denoted by $c^Z_{X,Y}$. Note that if $X$ and $Y$ are basic sets of $\mathcal{A}$ then $\underline{X}~\underline{Y}=\sum_{Z\in \mathcal{S}(\mathcal{A})}c^Z_{X,Y}\underline{Z}$. Therefore the numbers  $c^Z_{X,Y}$ are structure constants of $\mathcal{A}$ with respect to the basis $\{\underline{X}:\ X\in\mathcal{S}\}$.

A set $X \subseteq G$ is called an \emph{$\mathcal{A}$-set} if $\underline{X}\in \mathcal{A}$. A subgroup $H \leq G$ is called an \emph{$\mathcal{A}$-subgroup} if $H$ is an $\mathcal{A}$-set. With each $\mathcal{A}$-set $X$ one can naturally associate two $\mathcal{A}$-subgroups, namely $\langle X \rangle$ and $\rad(X)$. Let $L \unlhd U\leq G$. A section $U/L$ is called an \emph{$\mathcal{A}$-section} if $U$ and $L$ are $\mathcal{A}$-subgroups. If $S=U/L$ is an $\mathcal{A}$-section then the module
$$\mathcal{A}_S=Span_{\mathbb{Z}}\left\{\underline{X}^{\pi}:~X\in\mathcal{S}(\mathcal{A}),~X\subseteq U\right\},$$
where $\pi:U\rightarrow U/L$ is the canonical epimorphism, is an $S$-ring over $S$.

Let $K \leq \aut(G)$. Then  $\orb(K,G)$ forms a partition of  $G$ that defines an  $S$-ring $\mathcal{A}$ over $G$.  In this case  $\mathcal{A}$ is called \emph{cyclotomic} and denoted by $\cyc(K,G)$.

\subsection{Isomorphisms of $S$-rings}

Let $\mathcal{A}$ and $\mathcal{A}^{'}$ be $S$-rings over groups $G$  and  $G^{'}$ respectively. If there exists an isomorphism from  $\mathcal{A}$  to $\mathcal{A}^{'}$ we write $\mathcal{A}\cong\mathcal{A}^{'}$.  The group $\iso(\mathcal{A})$ of all isomorphisms from $\mathcal{A}$ onto itself has a normal subgroup
$$\aut(\mathcal{A})=\{f\in \iso(\mathcal{A}): R(X)^f=R(X)~\text{for every}~X\in \mathcal{S}(\mathcal{A})\}.$$
This subgroup is called the \emph{automorphism group} of $\mathcal{A}$. Note that $\aut(\mathcal{A})\geq G_{right}$. If $S$ is an $\mathcal{A}$-section then, obviously, $\aut(\mathcal{A})^S\leq\aut(\mathcal{A}_S)$. 

An \emph{algebraic isomorphism} from $\mathcal{A}$  to $\mathcal{A}^{'}$ is, in fact, a ring isomorphism of them. However, we define an algebraic isomorphism of $S$-rings in the following way which is more convenient for us. A bijection $\varphi:\mathcal{S}(\mathcal{A})\rightarrow\mathcal{S}(\mathcal{A}^{'})$ is said to be an \emph{algebraic isomorphism} from $\mathcal{A}$  to $\mathcal{A}^{'}$ if 
$$c_{X,Y}^Z=c_{X^{\varphi},Y^{\varphi}}^{Z^{\varphi}}$$ 
for all $X,Y,Z\in \mathcal{S}(\mathcal{A})$. The mapping $\underline{X}\rightarrow \underline{X}^{\varphi}$ is extended by linearity to the ring isomorphism of $\mathcal{A}$  and $\mathcal{A}^{'}$. If there exists an algebraic isomorphism from  $\mathcal{A}$  to $\mathcal{A}^{'}$ we write $\mathcal{A}\cong_{\alg}\mathcal{A}^{'}$. Every isomorphism $f$ of $S$-rings  preserves  the structure constants  and hence $f$ induces the algebraic isomorphism $\varphi_f$.

Let $\varphi:\mathcal{A}\rightarrow \mathcal{A}^{'}$ be an algebraic isomorphism. It is easy to see that $\varphi$ is extended to a bijection between  $\mathcal{A}$- and $\mathcal{A}^{'}$-sets and hence between  $\mathcal{A}$- and $\mathcal{A}^{'}$-sections. The images of an $\mathcal{A}$-set $X$ and an $\mathcal{A}$-section $S$ under the action of $\varphi$ are denoted by $X^{\varphi}$ and $S^{\varphi}$ respectively. If $S$ is an $\mathcal{A}$-section then  $\varphi$ induces the algebraic isomorphism $\varphi_S:\mathcal{A}_S\rightarrow \mathcal{A}^{'}_{S^{'}}$, where $S^{'}=S^{\varphi}$. The above bijection between the $\mathcal{A}$- and $\mathcal{A}^{'}$-sets is, in fact, an isomorphism of the corresponding lattices. It follows that
$$ \langle X^{\varphi} \rangle = \langle X \rangle ^{\varphi}~\text{and}~\rad(X^{\varphi})=\rad(X)^{\varphi}~$$for every  $\mathcal{A}$-set $X$. Since $c^{\{e\}}_{X,Y}=\delta_{Y,X^{-1}}|X|$ and $|X|=c^{\{e\}}_{X,X^{-1}}$ , where $X,Y\in \mathcal{S}(\mathcal{A})$ and $\delta_{X,X^{-1}}$ is the Kronecker delta, we conclude that $(X^{-1})^{\varphi}=(X^{\varphi})^{-1}$ and $|X|=|X^{\varphi}|$ for every $\mathcal{A}$-set $X$. In particular, $|G|=|G^{'}|$.

\begin{lemm}\cite[Lemma 2.1]{EP5}\label{uniq}
Let $\mathcal{A}$ and $\mathcal{A}^{'}$ be $S$-rings over groups $G$ and $G^{'}$ respectively. Let $\mathcal{B}$ be the $S$-ring generated by $\mathcal{A}$ and an element $\xi\in \mathbb{Z}G$ and $\mathcal{B}^{'}$  the $S$-ring generated by $\mathcal{A}^{'}$ and an element $\xi^{'}\in \mathbb{Z}G^{'}$.  Then given algebraic isomorphism  $\varphi:\mathcal{A}\rightarrow \mathcal{A}^{'}$ there is at most one algebraic isomorphism $\psi:\mathcal{B}\rightarrow \mathcal{B}^{'}$ extending $\varphi$ and such that $\xi^{\psi}=\xi^{'}$.
\end{lemm}	

Note that  for every group $G$ the $S$-ring of rank~2 over $G$ and $\mathbb{Z}G$ are separable with respect to the class of all groups. In the former case every basic set is singleton and hence every algebraic isomorphism is induced by an isomorphism in a natural way. In the latter case  there exists the unique algebraic isomorphism from the $S$-ring of rank~2 over $G$ to the $S$-ring of rank~2 over a given  group and this algebraic isomorphism is induced by every isomorphism.

A \emph{Cayley isomorphism} from  $\mathcal{A}$  to $\mathcal{A}^{'}$   is defined to be a group isomorphism $f:G\rightarrow G^{'}$ such that $\mathcal{S}(\mathcal{A})^f=\mathcal{S}(\mathcal{A}^{'})$. If there exists a Cayley isomorphism from  $\mathcal{A}$  to $\mathcal{A}^{'}$ we write $\mathcal{A}\cong_{\cay}\mathcal{A}^{'}$. Every Cayley isomorphism is a (combinatorial) isomorphism, however the converse statement is not true.

\subsection{Multiplier theorems}

Sets $X,Y\subseteq G$ are called \emph{rationally conjugate} if there exists $m\in \mathbb{Z}$ coprime to $|G|$ such that $Y=X^{(m)}$. The next two statements are known as the Schur theorems on multipliers (see \cite[Theorem~23.9]{Wi}).

\begin{lemm} \label{burn}
Let $\mathcal{A}$ be an $S$-ring over an abelian group  $G$. Then $X^{(m)}\in \mathcal{S}(\mathcal{A})$  for every  $X\in \mathcal{S}(\mathcal{A})$ and every  $m\in \mathbb{Z}$ coprime to $|G|$. Other words, every central element of $\aut(G)$  is a Cayley isomorphism from   $\mathcal{A}$ onto itself.
\end{lemm}

\begin{lemm} \label{sch}
Let $\mathcal{A}$ be an $S$-ring over an abelian group $G$, $p$  a prime divisor of $|G|$, and  $H=\{g\in G:g^p=e\}$. Then for every  $\mathcal{A}$-set $X$ the set $X^{[p]}=\{x^p:x\in~X,~|X\cap Hx|\not\equiv 0\mod p\}$ is an $\mathcal{A}$-set.
\end{lemm}

\subsection{Wreath and tensor products}
	
Let $\mathcal{A}$ be an $S$-ring over a group $G$ and $U/L$  an $\mathcal{A}$-section. The $S$-ring~$\mathcal{A}$ is called the \emph{$U/L$-wreath product}  if $L\trianglelefteq G$ and $L\leq\rad(X)$ for all basic sets $X$ outside~$U$. When the explicit indication of the section~$U/L$ is not important we use the term \emph{generalized wreath product}.  The $U/L$-wreath product is called \emph{nontrivial} or \emph{proper}  if $e\neq L$ and $U\neq G$. If $U=L$ we say that $\mathcal{A}$ is the \emph{wreath product} of  $\mathcal{A}_L$ and $\mathcal{A}_{G/L}$ and write $\mathcal{A}=\mathcal{A}_L\wr\mathcal{A}_{G/L}$.

\begin{lemm}\cite[Lemma 4.4]{Ry}\label{sepwr}
Let $\mathcal{A}$ be the $U/L$-wreath product over an abelian group $G$. Suppose that $\mathcal{A}_U$ and $\mathcal{A}_{G/L}$ are separable and $\aut(\mathcal{A}_U)^{U/L}=\aut(\mathcal{A}_{U/L})$. Then $\mathcal{A}$ is separable. In particular, the wreath product of two separable $S$-rings is separable.
\end{lemm}    

If $\mathcal{A}_1$ and $\mathcal{A}_2$ are $S$-rings over groups $G_1$ and $G_2$ respectively then the subring $\mathcal{A}=\mathcal{A}_1\otimes \mathcal{A}_2$ of the ring $\mathbb{Z}G_1\otimes\mathbb{Z}G_2=\mathbb{Z}G$, where $G=G_1\times G_2$, is an $S$-ring over the group $G$ with
$$\mathcal{S}(\mathcal{A})=\{X_1\times X_2: X_1\in\mathcal{S}(\mathcal{A}_1),\ X_2\in\mathcal{S}(\mathcal{A}_2)\}.$$
It is called the \emph{tensor product} of $\mathcal{A}_1$ and $\mathcal{A}_2$.

\begin{lemm}\cite[Lemma 2.3]{EKP}\label{tenspr}
Let $\mathcal{A}$ be an $S$-ring over an abelian group $G=G_1\times G_2$. Suppose that $G_1$ and $G_2$ are $\mathcal{A}$-subgroups. Then 

$(1)$ $X_{G_i}\in \mathcal{S}(\mathcal{A})$ for all $X\in \mathcal{S}(\mathcal{A})$ and $i=1,2;$

$(2)$ $\mathcal{A} \geq \mathcal{A}_{G_1}\otimes \mathcal{A}_{G_2}$, and the equality is attained whenever $\mathcal{A}_{G_i}=\mathbb{Z}G_i$ for some $i\in \{1,2\}$.
\end{lemm}

\begin{lemm} \label{schurtens}
The	 tensor product of two separable $S$-rings is separable.
\end{lemm}
	
\begin{proof}
Follows from~\cite[Theorem~1.20]{E}.
\end{proof}

\subsection{Subdirect product}

Let $U=\langle u \rangle$ and $V=\langle v \rangle$  cyclic groups and $|U|$ divides $|V|$. Then $V$ contains the unique subgroup $W$ of index~$|U|$. Let $\pi:V\rightarrow V/W$ be the canonical epimorphism and $\psi:U\rightarrow V/W$  an isomorphism.  We can form the subdirect product $A(U,V,\psi)$ of $U$ and $V$ in the following way:
$$A(U,V,\psi)=\{(x,y)\in U\times V| x^{\psi}=y^{\pi}\}.$$
The definition of $A(U,V,\psi)$ implies that
$$|A(U,V,\psi)|=|V|.~\eqno(1)$$
We say that the subdirect product of two groups is \emph{nontrivial} if it does not coincide with the direct product of these groups.

\section{$S$-rings over an abelian group of order $4p$}

Let $p$ be a prime. Put $E_1=\langle  a \rangle \times \langle b \rangle$, $E_2=\langle c \rangle$, and $P=\langle z \rangle$, where $|a|=|b|=2$, $|c|=4$, and $|z|=p$. Let $E\in \{E_1,E_2\}$ and $G=E \times P$. These notations are valid until the end of the paper. Throughout this section $\mathcal{A}$ is an $S$-ring over $G$.

\begin{lemm}\label{Sring0}
If $p=2$ and $E=E_1$ then one of the following statements holds:

$(1)$ $\mathcal{A}=\mathbb{Z}G$;

$(2)$ $\rk(\mathcal{A})=2$;

$(3)$ $\mathcal{A}$ is the tensor product of two $S$-rings over proper subgroups of $G$;

$(4)$ $\mathcal{A}$ is the wreath product of two $S$-rings over proper subgroups of $G$.
\end{lemm}

\begin{proof}
There are exactly nine $S$-rings over $G$ up to Cayley isomorphism. This can be checked with the help of the GAP package COCO2~\cite{GAP}. The statement of the lemma can be established by inspecting the above nine $S$-rings one after the other.
\end{proof}

From now on throughout this section we assume that $p\geq 3$.

\begin{lemm}\label{Sring1}
If $E$ or $P$ is not an $\mathcal{A}$-subgroup then one of the following statements holds:

$(1)$ $\rk(\mathcal{A})=2$;

$(2)$ $\mathcal{A}$ is the proper $U/L$-wreath product for some $\mathcal{A}$-section $U/L$ with $|U/L|\leq 2$;

$(3)$ $E=E_1$ and $\mathcal{A}=\mathcal{A}_H \otimes \mathcal{A}_L$, where $H$ is an $\mathcal{A}$-subgroup of order $2$, $L$ is an $\mathcal{A}$-subgroup of order $2p$, and $G=H\times L$.

\end{lemm}

\begin{proof}
Let $\mathcal{A}$ be an $S$-ring over $G$. Denote the maximal $\mathcal{A}$-subgroup in $E$ by $H$. Suppose that $H\neq E$.  Then \cite[Lemma~6.2]{EKP} implies that one of the following statements holds: (1) $\mathcal{A}=\mathcal{A}_{H}\wr \mathcal{A}_{G/H}$, where $\rk(\mathcal{A}_{G/H})=2$; (2) $\mathcal{A}$ is the $U/L$-wreath product, where  $P\leq L<G$ and $U=HL$. In the former case Statement~1 of the lemma holds whenever $H$ is trivial and Statement~2 of the lemma holds whenever $H$ is not trivial. In the latter case Statement~2 of the lemma holds whenever $U<G$. Suppose that $U=G$. Then $|H|=2$ and $G=H\times L$. This yields that $E=E_1\cong C_2\times C_2$. Clearly, $\mathcal{A}_H=\mathbb{Z}H$.  So $\mathcal{A}=\mathcal{A}_H \otimes \mathcal{A}_L$ by Statement~2 of Lemma~\ref{tenspr} and Statement~3 of the lemma holds.

The case when $P$ is not an $\mathcal{A}$-subgroup is dual to the case when $E$ is not an $\mathcal{A}$-subgroup in the sense of the duality theory of $S$-rings over an abelian group, see \cite[Section 2.2]{EP4}. So in this case the lemma follows from  \cite[Theorem~2.4, Statement~2 of Theorem~2.5]{EP4}.
\end{proof}

Further until the end of the section we assume that $E$ and $P$ are $\mathcal{A}$-subgroups.

\begin{lemm}\label{conj}
If $X,Y\in \mathcal{S}(\mathcal{A})$ and $X_E=Y_E$ then $X$ and $Y$ are rationally conjugate.
\end{lemm}

\begin{proof}
The statement of the lemma follows from Lemma~\ref{burn} because the group $1\times \aut(P)$ is contained in the center of $\aut(G)$ and  $\aut(P)$ acts transitively on $P^\#$.
\end{proof}

From \cite[Theorem~5.1]{EP4} it follows that $\mathcal{A}_P=\cyc(K,P)$ for some $K\leq \aut(P)$. Since $|P|=p$, the group $\aut(P)$ is cyclic and hence $K$ is also cyclic. Let  $\theta$ be a generator of $K$. It can be checked in a straightforward way  that $\mathcal{A}_E=\mathbb{Z}E$, or $\mathcal{A}_E=\mathbb{Z}C_2\wr\mathbb{Z}C_2$, or $\rk(\mathcal{A}_E)=2$. If $E=E_2$ and $\rk(\mathcal{A}_E)=2$ then $\mathcal{A}_E$ is not cyclotomic because in this case $E^\#\in \mathcal{S}(\mathcal{A}_E)$ and $E^\#$ contains elements of orders~2 and~4. A straightforward check implies that in other cases $\mathcal{A}_E\cong_{\cay}\cyc(\langle \sigma \rangle,E)$, where $\sigma\in \aut(E)$ is trivial or one of the automorphisms listed in Table~1.

\begin{center}

{\small
\begin{tabular}{|l|l|l|l|}
  \hline
  $E$ & $\sigma$ & $|\sigma|$ & $\mathcal{A}_E$    \\
  \hline
  $E_1$ & $\sigma_1:(a,b)\rightarrow (b,ab)$ & $3$ & $\rk(\mathcal{A}_E)=2$\\ \hline
  $E_1$ & $\sigma_2:(a,b)\rightarrow (b,a)$ & $2$ & $\mathbb{Z}C_2\wr \mathbb{Z}C_2$\\ \hline
  $E_2$ & $\sigma_3:c\rightarrow c^{-1}$ & $2$ & $\mathbb{Z}C_2\wr \mathbb{Z}C_2$\\ \hline
\end{tabular}
}

Table 1.

\end{center}

Suppose that $|\sigma|$ divides $|K|$. Denote the subgroup of $K$ of index $|\sigma|$  by $M$. Put 
$$\psi:\sigma^i\rightarrow M\theta^i,~i=0,\ldots,|\sigma|-1.$$ 
Clearly, $\psi$ is an isomorphism from $\langle \sigma \rangle$ to $K/M$.

\begin{lemm}\label{Sring2}
If $\mathcal{A}\neq \mathcal{A}_E\otimes \mathcal{A}_P$  then $\mathcal{A}_E\cong_{\cay}\cyc(\langle \sigma \rangle,E)$, $|\sigma|$ divides $|K|$, and $\mathcal{A}\cong_{\cay}\cyc(A(\langle \sigma \rangle,K,\psi),G)$, where $\sigma\in \aut(E)$ is one of the automorphisms listed in Table~1.
\end{lemm}

\begin{proof}
If $\mathcal{A}_E=\mathbb{Z}E$ then $\mathcal{A}=\mathcal{A}_E\otimes \mathcal{A}_P$ by Statement~2 of Lemma~\ref{tenspr} and we obtain a contradiction with the assumption of the lemma. So 
$$\mathcal{A}_E=\mathbb{Z}C_2\wr\mathbb{Z}C_2~\text{or}~\rk(\mathcal{A}_E)=2.$$ 
Prove that $\mathcal{A}=\cyc(A^{'},G)$ for some $A^{'}\leq \aut(G)$. If $E=E_1$ then this follows from \cite[p.15-16]{EKP}. Let $E=E_2$. Note that $\mathcal{A}$ is not the proper generalized wreath product of two $S$-rings because $E$ and $P$ are $\mathcal{A}$-subgroups. Since $\mathcal{A}\neq \mathcal{A}_E\otimes \mathcal{A}_P$, we conclude by \cite[Theorem~4.1,Theorem~4.2]{EP3} that $\mathcal{A}=\cyc(A^{'},G)$ for some $A^{'}\leq \aut(G)$. 

Clearly, $\mathcal{A}_E$ is  cyclotomic. So we may assume that $\mathcal{A}_E=\cyc(\langle \sigma \rangle,E)$, where $\sigma\in\{\sigma_1,\sigma_2,\sigma_3\}$. Since  $\mathcal{A}_E=\cyc((A^{'})^E,E)$, $\mathcal{A}_P=\cyc((A^{'})^P,P)$, and $\mathcal{A}\neq \mathcal{A}_E\otimes \mathcal{A}_P$, the group $A^{'}$ is the nontrivial subdirect product of $\langle \sigma\rangle$ and $K$. If $|K|$ is not divisible by~$|\sigma|$ then there are no nontrivial subdirect products of  $\langle \sigma \rangle$ and $K$ because $|\sigma|\in\{2,3\}$. So $|\sigma|$ divides $|K|$. If $|\sigma|=2$ then $A(\langle \sigma \rangle,K,\psi)$ is the unique nontrivial subdirect product of $\langle \sigma \rangle$ and $K$. Therefore $A^{'}=A(\langle \sigma \rangle,K,\psi)$ and we are done. 

Suppose that $|\sigma|=3$. Then $\sigma=\sigma_1$, $E=E_1$, and $\rk(\mathcal{A}_E)=2$. In this case there are exactly two nontrivial subdirect products of $\langle \sigma \rangle$ and $K$:
$$A(\langle \sigma \rangle,K,\psi)~\text{and}~A(\langle \sigma \rangle,K,\xi),$$ 
where $\xi:\sigma^i\rightarrow M\theta^{-i},~i=0,1,2$. So $A^{'}\in\{A(\langle \sigma \rangle,K,\psi),A(\langle \sigma \rangle,K,\xi)\}$. The straightforward check implies that for every involution $\tau\in \aut(E)$ the automorphism $\tau\times 1\in \aut(E)\times \aut(P)$ is a Cayley isomorphism from $A(\langle \sigma \rangle,K,\psi)$ to $A(\langle \sigma \rangle,K,\xi)$. Therefore $\mathcal{A}\cong_{\cay}\cyc(A(\langle \sigma \rangle,K,\psi),G)$ and the statement of the lemma holds.
\end{proof}

Given a group $K\leq \aut(P)$ put $\mathcal{A}_i(K)=\cyc(A(\langle \sigma_i \rangle,K,\psi),G)$, where $i\in\{1,2,3\}$ and $\sigma_i$ is from Table~1. If $K_1,K_2\leq \aut(P)$ and $K_1\neq K_2$ then $\cyc(K_1,P)\ncong_{\alg}\cyc(K_2,P)$ and hence $\mathcal{A}_i(K_1)\ncong_{\alg}\mathcal{A}_j(K_2)$ for all $i,j\in\{1,2,3\}$. 

\begin{lemm}\label{nonisom}
Let $K\leq \aut(P)$. Then  $\mathcal{A}_i(K)\ncong_{\alg}\mathcal{A}_j(K)$ whenever $i\neq j$. 
\end{lemm}

\begin{proof}
Note that for every $i\in\{1,2,3\}$ the group $E$ is the unique $\mathcal{A}_i(K)$-subgroup of order~$4$, $\rk(\mathcal{A}_1(K)_E)=2$, and $\rk(\mathcal{A}_2(K)_E)=\rk(\mathcal{A}_3(K)_E)=3$. So $\mathcal{A}_1(K)\ncong_{\alg}\mathcal{A}_2(K)$ and $\mathcal{A}_1(K)\ncong_{\alg}\mathcal{A}_3(K)$. 

Now let $\mathcal{A}$ and $\mathcal{A}^{'}$ be $S$-rings over the groups $G=E_1\times P$ and $G^{'}=E_2\times P$ respectively, $\mathcal{A}\cong_{\cay}\mathcal{A}_2(K)$, and $\mathcal{A}^{'}\cong_{\cay}\mathcal{A}_3(K)$. Assume that $\mathcal{A}\cong_{\alg}\mathcal{A}^{'}$ and $\varphi:\mathcal{A}\rightarrow \mathcal{A}^{'}$ is an algebraic isomorphism. Then $E_1^{\varphi}$ and $P^{\varphi}$ are  $\mathcal{A}^{'}$-subgroups of orders $4$ and $p$ respectively. So $E_1^{\varphi}=E_2$ and $P^{\varphi}=P$. Let $X\in \mathcal{S}(\mathcal{A})$ such that $X\nsubseteq E_1$ and $X_{E_1}=\{a,b\}$. Then $X=aX_1\cup bX_2$, where $X_1,X_2\subseteq P$. From Statement~1 of Lemma~\ref{tenspr} it follows that $X_P=X_1\cup X_2\in \mathcal{S}(\mathcal{A}_P)$. Lemma~\ref{burn} implies that $Y=X_P^{(2)}\in \mathcal{S}(\mathcal{A}_P)$. Clearly,
$$\underline{X}^2=\underline{X_1}^2+\underline{X_2}^2+2ab\underline{X_1}~\underline{X_2}.$$
Therefore
$$c_{X,X}^{Y}~\text{is odd}.~\eqno(2)$$
Note that $\langle X \rangle=G$. So $\langle X^{\varphi} \rangle=G^{'}$ by the properties of an algebraic isomorphism and hence $X^{\varphi}=cX_1^{'}\cup c^{-1}X_2^{'}$, where $X_1^{'},X_2^{'}\subseteq P$. It is easy to see that
$$(\underline{X}^{\varphi})^2=2\underline{X_1}^{'}~\underline{X_2}^{'}+c^2((\underline{X_1}^{'})^2+(\underline{X_2}^{'})^2).$$
Therefore for every $Y^{'}\in\mathcal{S}(\mathcal{A}^{'}_P)$ the number $c_{X^{\varphi},X^{\varphi}}^{Y^{'}}$ is even. On the other hand, $(Y)^{\varphi}\in\mathcal{S}(\mathcal{A}^{'}_P)$ and (2) yields that $c_{X^{\varphi},X^{\varphi}}^{Y^{\varphi}}=c_{X,X}^{Y}$ is odd, a contradiction. Thus, $\mathcal{A}\ncong_{\alg}\mathcal{A}^{'}$ and the lemma is proved.
\end{proof}

Let $\mathcal{A}\cong_{\cay} \mathcal{A}_i(K)$ for some $K\leq \aut(P)$ and $i\in\{1,2,3\}$.  From the discription of $\sigma_i$  given in Table~1 it follows that the group $\langle \sigma_i \rangle$ has the unique regular orbit $O\in\mathcal{S}(\mathcal{A}_E)$. Following~\cite{EKP}, we say that $X\in \mathcal{S}(\mathcal{A})$ is a \emph{highest} basic set if $X$ lies outside $E\cup P$ and $X_E=O$. Highest basic sets of $\mathcal{A}$ exist. Indeed, if $X\in \mathcal{S}(\mathcal{A})$ such that $gx\in X$, where $g\in O$ and $x\in P^\#$, then $X$ lies outside $E\cup P$ and $X_E=O$ by Statement~1 of Lemma~\ref{tenspr}. So $X$ is highest.

\begin{lemm}\label{generate}
Suppose that $\mathcal{A}\cong_{\cay} \mathcal{A}_i(K)$ for some $K\leq \aut(P)$ and $i\in\{1,2,3\}$. Then the following statements hold:

$(1)$ a basic set $X$ of $\mathcal{A}$ is highest if and only if $\langle X \rangle=G$;

$(2)$ if $X\in \mathcal{S}(\mathcal{A})$ is highest then $\mathcal{A}=\langle \underline{X} \rangle$.
\end{lemm}

\begin{proof}
Let $O\in \mathcal{S}(\mathcal{A}_E)$ be a regular orbit of $\langle \sigma_i \rangle$. The straightforward check implies that $\langle Y \rangle=E$ for $Y\in \mathcal{S}(\mathcal{A}_E)$ if and only if $Y=O$. So $\langle X \rangle =G$ for $X\in\mathcal{S}(\mathcal{A})$ if and only if $X$ is highest and Statement~1 of the lemma is proved. 

Now let $X$ be a highest basic set of $\mathcal{A}$ and $\mathcal{B}=\langle \underline{X} \rangle$. Prove that $\mathcal{A}=\mathcal{B}$. Clearly, $\mathcal{A}\geq\mathcal{B}$. On the one hand, $X$ is a union of some basic sets of $\mathcal{B}$ because $\underline{X}\in \mathcal{B}$. On the other hand, $X$ is contained in some basic set of $\mathcal{B}$ because $\mathcal{A}\geq\mathcal{B}$. So $X\in \mathcal{S}(\mathcal{B})$. 

From   Lemma~\ref{Sring2} it follows that: (1) $|xE\cap X|=1$; (2) $|xP\cap X|=|K|/3$ if $i=1$ and $|xP\cap X|=|K|/2$ if $i\in\{2,3\}$. Therefore $O=X^{[p]}$ and $X^{[2]}$ are $\mathcal{B}$-sets by Lemma~\ref{sch}. So $E=\langle O\rangle$ and $P=\langle X^{[2]}\rangle$ are $\mathcal{B}$-subgroups. Statement~1 of Lemma~\ref{tenspr} implies that $X_E,X_P\in\mathcal{S}(\mathcal{B})$ and hence 
$$\mathcal{B}_E=\mathcal{A}_E~\text{and}~\mathcal{B}_P=\mathcal{A}_P.~\eqno(3)$$
Since $X\in \mathcal{S}(\mathcal{B})$ and $X\neq X_E\times X_P$, we obtain $\mathcal{B}\neq \mathcal{B}_E\otimes \mathcal{B}_P$. Therefore Lemma~\ref{Sring2} holds for $\mathcal{B}$. The set $X$ is also a highest basic set of $\mathcal{B}$. Let $Y\in \mathcal{S}(\mathcal{B})$ outside $E\cup P$. If $Y$ is highest then $Y=X^{(m)}$ for some $m$ coprime to $|G|$ by Lemma~\ref{conj}. So $Y\in \mathcal{S}(\mathcal{A})$ by Lemma~\ref{burn}. If $Y$ is nonhighest then $Y=Y_E\times Y_P$. Due to (3), we conclude that $Y\in \mathcal{S}(\mathcal{A})$ . Thus $\mathcal{B}=\mathcal{A}$ and Statement~2 of the lemma is proved.
\end{proof}

\section{Proof of Theorem~\ref{main}}

In this section we keep the notations from the previous one. We start the proof with the following lemma which implies that every proper section of $G$ is separable.

\begin{lemm}\label{cycp}
The groups $C_2\times C_2$, $C_p$, and $C_{2p}$, where $p$ is a prime, are separable.
\end{lemm}

\begin{proof}
The groups $C_2\times C_2$, $C_p$, and $C_4$ are separable by~\cite[Theorem 1]{Ry}, \cite[Theorem 1.3]{EP5}, and \cite[Lemma 5.5]{Ry} respectively. Suppose that $p\neq 2$. Let $H=H_1\times H_2$, where $H_1=C_2$ and $H_2=C_p$, and $\mathcal{C}$  an $S$-ring over $H$. If $\mathcal{C}$ is cyclotomic then $H_1$ and $H_2$ are $\mathcal{C}$-subgroups. Clearly, $\mathcal{C}_{H_1}=\mathbb{Z}H_1$ and hence $\mathcal{C}=\mathcal{C}_{H_1}\otimes \mathcal{C}_{H_2}$ by Statement~2 of Lemma~\ref{tenspr}. Now  applying \cite[Theorem 4.1, Theorem 4.2]{EP3} to $H$ and $\mathcal{C}$, we obtain that one of the following statements holds: (1) $\rk(\mathcal{C})=2$; (2) $\mathcal{C}=\mathbb{Z}H$; (3) $\mathcal{C}=\mathcal{C}_{H_i}\wr \mathcal{C}_{H/H_i}$, $i\in\{1,2\}$; (4) $\mathcal{C}=\mathcal{C}_{H_1}\otimes \mathcal{C}_{H_2}$. In the first and the second cases $\mathcal{C}$ is obviously separable. In the third case $\mathcal{C}$ is separable by Lemma~\ref{sepwr}. In the fourth case $\mathcal{C}$ is separable by Lemma~\ref{schurtens}. Thus $H=C_{2p}$ is separable and the lemma is proved.
\end{proof}

Let $\mathcal{A}$ be an $S$-ring over $G$. Prove that $\mathcal{A}$ is separable. If $p=2$ then $G\cong C_8$, or $G\cong C_4\times C_2$, or $G\cong C_2^3$. In the first case $\mathcal{A}$ is separable by \cite[Lemma 5.5]{Ry}. In the second case $\mathcal{A}$ is separable by \cite[Theorem 1]{Ry}. In the third case $\mathcal{A}$ is separable by Lemma~\ref{Sring0}, Lemma~\ref{cycp}, Lemma~\ref{sepwr}, and Lemma~\ref{schurtens}.

Now let $p\geq 3$. Then Lemma~\ref{Sring1} and Lemma~\ref{Sring2} yield that one of the statements of Lemma~\ref{Sring1}  holds for $\mathcal{A}$, or $\mathcal{A}=\mathcal{A}_E\otimes \mathcal{A}_P$, or $\mathcal{A}\cong_{\cay} \mathcal{A}_i(K)$ for some $K\leq \aut(P)$ and $i\in\{1,2,3\}$. If $\rk(\mathcal{A})=2$ then, obviously, $\mathcal{A}$ is separable. Suppose that Statement~2 of Lemma~\ref{Sring1} holds for $\mathcal{A}$. In this case $\mathcal{A}$ is the proper $U/L$-wreath product, where $U/L$ is an $\mathcal{A}$-section with $|U/L|\leq 2$. Check that the conditions from Lemma~\ref{sepwr} holds for $\mathcal{A}$. Lemma~\ref{cycp} implies that the $S$-rings $\mathcal{A}_U$ and $\mathcal{A}_{G/L}$ are separable.  On the one hand, $\aut(\mathcal{A}_U)^{U/L}\leq \aut(\mathcal{A}_{U/L})$. On the other hand, since $|U/L|\leq 2$, we obtain that
$$\aut(\mathcal{A}_U)^{U/L}\geq (U_{right})^{U/L}=(U/L)_{right}=\aut(\mathcal{A}_{U/L}).$$
Therefore $\aut(\mathcal{A}_U)^{U/L}=\aut(\mathcal{A}_{U/L})$ and $\mathcal{A}$ is separable by Lemma~\ref{sepwr}. If Statement~3 of Lemma~\ref{Sring1} holds for $\mathcal{A}$ or $\mathcal{A}=\mathcal{A}_E\otimes \mathcal{A}_P$ then $\mathcal{A}$ is separable by Lemma~\ref{cycp} and Lemma~\ref{schurtens}. 

It remains to consider only the case when $\mathcal{A}\cong_{\cay}\mathcal{A}_i(K)=\cyc(A(\langle \sigma_i \rangle,K,\psi),G)$ for some $K\leq \aut(P)$ and $i\in\{1,2,3\}$. In this case $\mathcal{A}_P=\cyc(K,P)$ and $\mathcal{A}_E\neq \mathbb{Z}E$. Every basic set of $\mathcal{A}_P$ has cardinality $|K|$ because $K$ acts semiregularly on $P^\#$. From~$(1)$ it follows that $|A(\langle \sigma_i \rangle,K,\psi)|=|K|$ and hence every basic set of $\mathcal{A}$ has  cardinality at most~$|K|$. 

Let $\mathcal{A}^{'}$ be an $S$-ring over an abelian group  $G^{'}$ and $\varphi:\mathcal{A}\rightarrow \mathcal{A}^{'}$  an algebraic isomorphism.

\begin{lemm}\label{cayleyisom}
$\mathcal{A}^{'}\cong_{\cay}\mathcal{A}$.
\end{lemm}

\begin{proof}
Clearly, $|G^{'}|=4p$, $E^{'}=E^{\varphi}$ is an $\mathcal{A}^{'}$-subgroup of order~$4$, and $P^{'}=P^{\varphi}$ is an $\mathcal{A}^{'}$-subgroup of order~$p$. From the properties of an algebraic isomorphism it follows that every basic set of $\mathcal{A}^{'}_{P^{'}}$ has cardinality $|K|$. Since  $\mathcal{A}_E\neq \mathbb{Z}E$, we have  $\mathcal{A}^{'}_{E^{'}}\neq \mathbb{Z}E^{'}$. Assume that $\mathcal{A}^{'}=\mathcal{A}^{'}_{E^{'}}\otimes\mathcal{A}^{'}_{P^{'}}$. Then  there exists  $Z^{'}\in \mathcal{S}(\mathcal{A}^{'})$ with $|Z^{'}|\geq 2|K|$ because $\mathcal{A}^{'}_{E^{'}}\neq \mathbb{Z}E^{'}$. Since $\varphi$ is an algebraic isomorphism, $Z^{'}\in \mathcal{S}(\mathcal{A}^{'})$ and $|(Z^{'})^{\varphi^{-1}}|\geq 2|K|$. We obtain a contradiction because every basic set of $\mathcal{A}$ has a cardinality at most~$|K|$. Thus $\mathcal{A}^{'}\neq\mathcal{A}_{E^{'}}\otimes\mathcal{A}_{P^{'}}$. So $\mathcal{A}^{'}\cong_{\cay}\mathcal{A}_j(K)$ for some $j\in\{1,2,3\}$ by Lemma~\ref{Sring2}. If $i\neq j$ then $\mathcal{A}^{'}\ncong_{\alg}\mathcal{A}$ by Lemma~\ref{nonisom} that contradicts to our assumption. Therefore $i=j$ and $\mathcal{A}^{'}\cong_{\cay}\mathcal{A}$. 
\end{proof}

\begin{lemm}
The algebraic isomorphism $\varphi$ is induced by a Cayley isomorphism.
\end{lemm}

\begin{proof}
Lemma~\ref{cayleyisom} implies that there exists a Cayley isomorphism $f$ from $\mathcal{A}$ to $\mathcal{A}^{'}$. Let  $X\in \mathcal{S}(\mathcal{A})$ be a highest basic set. Then $\langle X \rangle=G$ by Statement~1 of Lemma~\ref{generate}. So $\langle X^{\varphi} \rangle=G^{'}$ and $\langle X^f \rangle=G^{'}$ by the properties of an algebraic isomorphism. Due to Statement~1 of Lemma~\ref{generate}, the sets $X^{\varphi}$ and $X^f$ are highest basic sets of $\mathcal{A}^{'}$. Lemma~\ref{conj} yields that $X^{\varphi}$ and $X^f$ are rationally conjugate. Therefore there exists a Cayley isomorphism $f_1$ from $\mathcal{A}^{'}$ onto itself such that $X^{ff_1}=X^{\varphi}$. The Cayley isomorphism $ff_1$ from $\mathcal{A}$ to $\mathcal{A}^{'}$ induces the algebraic isomorphism $\varphi_{ff_1}$ and $X^{\varphi_{ff_1}}=X^{ff_1}=X^{\varphi}$.  Note that $\mathcal{A}=\langle \underline{X} \rangle$ and $\mathcal{A}^{'}=\langle \underline{X}^{\varphi} \rangle$ by Statement~2 of Lemma~\ref{generate}. Thus $\varphi=\varphi_{ff_1}$ by Lemma~\ref{uniq}.
\end{proof}

We proved that if $\mathcal{A}\cong_{\cay}\mathcal{A}_i(K)$ for some $K\leq \aut(P)$ and $i\in\{1,2,3\}$ then every algebraic isomorphism of $\mathcal{A}$ is induced by a Cayley isomorphism. So $\mathcal{A}$ is separable in this case and the proof of Theorem~\ref{main} is complete.

\end{document}